\documentclass[12pt,oneside]{article}
\usepackage[OT4]{fontenc}
\usepackage{a4}
\usepackage{amsmath}
\usepackage{amsfonts}
\usepackage{color}
\usepackage{amssymb}
\usepackage{tikz}
\usepackage{amsthm}
\usepackage{epsfig}

\usepackage{verbatim}
\usepackage[utf8]{inputenc}
\newtheorem{thm}{Theorem}
\newtheorem{lem}[thm]{Lemma}

\newtheorem{obs}[thm]{Observation}

\theoremstyle{remark} 

\newtheorem{ex}[thm]{Example}

\title{Coronas and domination subdivision number of a~graph}
\author{M. Dettlaff$^{\dagger}$, M. Lema\'nska$^{\dagger}$, J. Topp$^{\ddagger}$, P. \.Zyli\'nski$^{\ddagger}$
\\
\\
$^{\dagger}${\small Faculty of Applied Physics and Mathematics } \\ {\small Gda\'nsk University of Technology,  80-233 Gda\'nsk, Poland }\\ {\small {\tt \{mdettlaff,magda\}\@@mif.pg.gda.pl}}
\\
$^{\ddagger}${\small Faculty of Mathematics, Physics and Informatics 
}\\
{\small  University of Gda\'nsk, 80-952 Gda\'nsk, Poland} \\{\small {\tt \{jtopp,zylinski\}\@@inf.ug.edu.pl}}}

\date{}

\begin{document}
\maketitle
\begin{abstract} 
\noindent In this paper, 
for a graph $G$
and a family of partitions $\mathcal{P}$ of  vertex neighborhoods of  $G$, we define the general corona $G \circ \mathcal{P}$ of $G$.  
Among several properties of this new operation, 
we focus on application general coronas to a~new kind of characterization 
of trees with the domination subdivision number equal to 3.
\end{abstract}



\section{Introduction}
In this paper we follow the notation and terminology of \cite{fund}. Let $G = (V(G), E(G))$ be a (finite, simple, undirected) graph of order $n = |V (G)|$. 
For a vertex $v$ of $G$, its {\em neighborhood}, denoted by $N_G(v)$, is the set of all vertices adjacent to $v$. The cardinality of $N_G(v)$, denoted by $d_G(v)$, is called the {\em degree} of $v$.
%
%
%
A vertex $v$ is a {\em leaf}{\/} of $G$ if  $d_G(v)=1$. Every neighbor of a leaf is called  a {\em support} vertex. A~{\em strong support} vertex is a vertex adjacent
to at least two leaves. 


A subset $D$ of $V(G)$ is said to be \emph{dominating} in $G$ if every vertex belonging to $V(G)-D$ has at least one neighbor in $D.$ The cardinality of a~smallest dominating set in $G$, denoted by $\gamma (G)$, is called the {\it domination number} of $G$.  
A subset $S$ of vertices in $G$ is called a $2$-{\it packing} if every two distinct vertices belonging to $S$ are at distance greater then $2$.  

The {\em corona} of graphs $G_1$ and $G_2$ 
is a graph $G = G_1 \circ G_2$ resulting from the disjoint union of $G_1$ and $|V (G_1)|$ 
copies of $G_2$ in which each vertex~$v$ of $G_1$ is adjacent to all vertices of the copy of $G_2$ corresponding to~$v$. 

For a graph $G$, the {\it subdivision} of an edge $e=uv$ with a new vertex $x$ is an operation which leads to a graph~$G'$ with $V(G')=V(G)\cup \{x\}$ and $E(G')=(E(G)-\{uv\})\cup \{ux,xv\}$. 
The graph obtained from $G$ by  the replacing every edge $e=uv$ with a~path $(u,x_1,x_2,v)$, where $x_1$ and $x_2$ are new vertices, is called the $2$-{\it subdivision} of $G$ and is denoted by $S_2(G)$.

For a graph $G$ and a family $\mathcal{P}=\{\mathcal{P}(v): v\in V(G)\}$, where $\mathcal{P}(v)$ is a partition of the set $N_G(v)$, by $G\circ \mathcal{P}$ we denote the graph in which

$$V(G\circ \mathcal{P})=\{(v,1)\colon v\in V(G)\}\cup \bigcup_{v\in V(G)} \{(v, A)\colon A\in \mathcal{P}(v)\}$$
and
$$E(G\circ \mathcal{P})= \!\!\bigcup_{v\in V(G)}\!\! \{(v,1)(v,A)\colon \! A\in \mathcal{P}(v)\}\cup \!\!\bigcup_{uv\in E(G)}\!\!\{(v,A)(u,B)\colon \! (u\in A)  \wedge  (v\in B) \}.$$

\noindent The family $\mathcal{P}$ is called a vertex neighborhood partition of $G$ and the graph $G\circ \mathcal{P}$ is called a $\mathcal{P}$-{\it corona} (or shortly {\it general corona}) of $G$. 
The set $\{(v,1)\colon v\in V(G)\}$ of vertices of $G\circ \mathcal{P}$ is denoted by $Ext(G\circ \mathcal{P})$  and its elements are called the {\it external} vertices. Those vertices of $G\circ \mathcal{P}$ which are not external, are said to be {\it internal}.

\begin{figure}[ht] 
\centerline{\epsfxsize=4in \epsffile{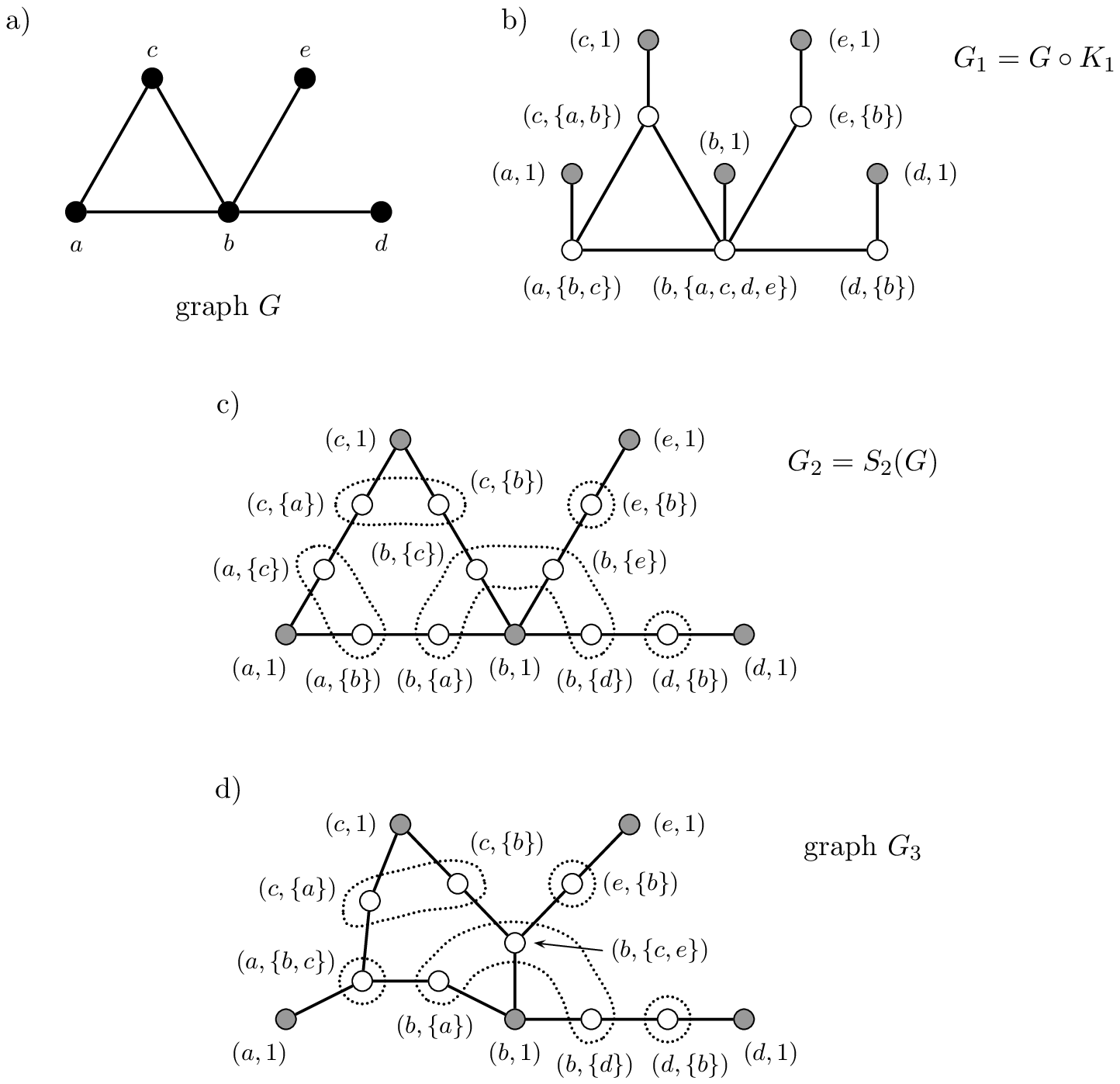}}
\caption{Graph $G$  and its exemplary coronas.}
\label{rys1}
\end{figure}

\begin{ex}\label{ex1}
Let $G$ be the graph shown in Figure \ref{rys1}(a) and let $\mathcal{P}=\{\mathcal{P}(a), \mathcal{P}(b), \mathcal{P}(c),$ $ \mathcal{P}(d), \mathcal{P}(e)\}$, where  $\mathcal{P}(a)=\{N_G(a)\}=\{\{b,c\}\},$ 
$\mathcal{P}(b)=\{N_G(b)\}=\{\{a,c,d,e\}\},$ $\mathcal{P}(c)=\{N_G(c)\}=\{\{a,b\}\},$ $\mathcal{P}(d)=\{N_G(d)\}=\{\{b\}\},$
$\mathcal{P}(e)=\{N_G(e)\}=\{\{b\}\}$. Then the $\mathcal{P}$-corona $G\circ \mathcal{P}$ is the graph $G_1$ given in Figure~\ref{rys1}(b) and in fact it is the corona $G\circ K_1$.      

Now  if $\mathcal{P}=\{\mathcal{P}(v)\colon v\in V(G)\}$ and $\mathcal{P}(v)$ is the family of all $1$-element subsets of $N_G(v)$, that is   $\mathcal{P}(a)=\{\{b\},\{c\}\},$ 
$\mathcal{P}(b)=\{\{a\},\{c\},\{d\},\{e\}\},$ $\mathcal{P}(c)=\{\{a\},\{b\}\},$ $\mathcal{P}(d)=\{ \{b\}\},$ $\mathcal{P}(e)=\{\{b\}\}$, then $G\circ \mathcal{P}$ is the graph $G_2$ shown in Figure~\ref{rys1}(c) and in this case it is the $2$-subdivision $S_2(G)$ of $G$.

Finally, let us consider -- for an example -- the case where $\mathcal{P}=\{\mathcal{P}(v)\colon v\in V(G)\}$ and $\mathcal{P}(a)=\{\{b,c\}\},$ 
$\mathcal{P}(b)=\{\{a\},\{c,e\},\{d\}\},$ $\mathcal{P}(c)=\{\{a\},\{b\}\},$ $\mathcal{P}(d)=\{\{b\}\},$ $\mathcal{P}(e)=\{\{b\}\}$. In this case  $G\circ \mathcal{P}$ is the graph $G_3$ shown in Figure \ref{rys1}(d). This graph is an example of possible general coronas of $G$ which are ``between'' the corona $G\circ K_1$ and the $2$-subdivision $S_2(G)$. 

\end{ex}

From the definition of general corona  it easily follows (what we have seen in the above  example) that: 
\begin{itemize}
\item [a)] if $\mathcal{P}(v)=\{N(v)\}$ for every $v\in V(G),$  then $G\circ \mathcal{P}$ is the corona $G\circ K_1$ (and the vertices of $G$ are internal vertices in $G\circ K_1$);
\item[b)] if $\mathcal{P}(v)=\{ \{u \}\colon u \in N_G(v) \}$  for every $v\in V(G),$   then $G\circ \mathcal{P}$  is the $2$-subdivision $S_2(G)$ (and the vertices of $G$ are external vertices of $S_2(G)$). 
\end{itemize}

Let $H$ be a subgraph of a graph $G$. The {\it contraction} of $H$ to a vertex is the replacement of $H$ by a single vertex $k$. Each edge that joined a vertex $v\in V(G)-V(H)$ to a vertex in $H$ is replaced by an edge with endpoints $v$ and~$k$.
 
Let $\mathcal{P}=\{\mathcal{P}(v)\colon v\in V(G)\}$ and $\mathcal{P}'=\{\mathcal{P}'(v)\colon v\in V(G)\}$ be two vertex neighborhood partitions of $G$. We say that $\mathcal{P}'$ is a~{\it refinement} of $\mathcal{P}$ and write $\mathcal{P}'\prec \mathcal{P}$ if for every $v\in V(G)$ and every $A\in \mathcal{P}'(v)$ there exists $B\in \mathcal{P}(v)$ such that $A\subseteq B$. If $\mathcal{P}'\prec \mathcal{P}$, then the general corona $G\circ \mathcal{P}'$ is said to be {\it refinement} of $G\circ \mathcal{P}$. In this case we write $G\circ \mathcal{P}'\prec G\circ \mathcal{P}$ and say that $G\circ \mathcal{P}'$ has been obtained from $G\circ \mathcal{P}$ by {\it splitting} some of its internal vertices. On the other hand, $G\circ \mathcal{P}$ can be obtained from $G\circ \mathcal{P}'$ contracting some of its internal vertices. For example,  $G_2$ from Figure~\ref{rys1} is refinement of $G_3$ and $G_3$ is refinement of $G_1$, so $G_2\prec G_3\prec G_1$.
Notice that in general, a graph $G\circ \mathcal{P}$ can be treated as a graph obtained from corona $G\circ K_1,$ where every support vertex $v$ is splitted according to the partition $\mathcal{P}(v)$ of $N_G(v).$  Let us again consider the graphs $G$, $G_1$, $G_2$ and $G_3$ from  Figure~\ref{rys1}.  The graph $G_2= S_2(G)$ can be obtained from  $G\circ K_1$ by splitting  support vertex into maximum possible number of vertices. Moreover, if in $G\circ K_1$ we split the vertex $(c,\{a,b\})$ into two vertices $(c,\{a\})$ and $(c,\{b\})$, the vertex $(b,\{a,c,d,e\})$ into three vertices: $(b,\{a\})$, $(b,\{c,e\})$, $(b,\{d\})$, and we leave other support vertices unsplitted, then we obtain $G_3$. On the other hand, $G_3$ can be obtained from $G_2=S_2(G)$ contracting $(a,\{c\})$ and $(a,\{b\})$, and also $(b,\{c\})$ and $(b,\{e\})$\

The contraction (splitting) of internal vertices is called an {\it internal contraction} ({\it  splitting}). We have  the following observations.

\begin{obs}\label{nowa}Let $T$ be a tree with at least three vertices. Then the following properties are equivalent:
\begin{itemize}
\item[$1.$] $T$ is a general corona of a tree.
\item[$2.$] There exists a tree $T'$ such that $T$ is obtained from the $2$-subdivision $S_2(T')$ by a~sequence of internal contractions.
\item[$3.$] There exists a tree $T'$ such that $T$ is obtained from the corona $T'\circ K_1$ by a~sequence of internal splittings.\qed
\end{itemize}
\end{obs}

\begin{obs}\label{ext}
If $G$ is a general corona of a tree, then $Ext(G)$ is a dominating $2$-packing of $G$ containing all leaves of $G$.
\end{obs}

\begin{proof}It follows from the following three facts: The distance between any two external vertices of $G$ is at least three. Next, every internal vertex of $G$ is adjacent to an external vertex. Finally, every leaf of $G$ belongs to $Ext(G)$.
\end{proof}
\begin{obs}\label{sklej_w}
Let $G$ and $H$ be general coronas of some trees. If they share only one vertex which is an external vertex in each of them, then $G\cup H$ is a general corona. 
\end{obs}
\begin{proof}Assume that $G$ and $H$ are general coronas of some trees $T_1$ and $T_2$, say $G=T_1\circ \mathcal{P}_1$ and $H=T_2\circ \mathcal{P}_2$ for some neighborhood partitions $\mathcal{P}_1$ and $\mathcal{P}_2$ of $T_1$ and $T_2$, respectively. Let $(v,1)$ be the only common external vertex of $G$ and $H$. Then the trees $T_1$ and $T_2$ share only $v$ and the union $T=T_1\cup T_2$ is a~tree. Now, let $\mathcal{P}$ be the family $\{\mathcal{P}(x)\colon x\in V(T)\}$, where  $\mathcal{P}(v)=\mathcal{P}_1(v)\cup \mathcal{P}_2(v)$, $\mathcal{P}(x)=\mathcal{P}_1(x)$ for $x\in V(T_1)-\{v\}$, and $\mathcal{P}(x)=\mathcal{P}_2(x)$ for $x\in V(T_2)-\{v\}$. Then $G\cup H$ is a~$\mathcal{P}$-corona of $T$, that is, $G\cup H=T\circ \mathcal{P}$, see Figure~\ref{rys_sklej_w}.
\end{proof}

\begin{figure}[htp] 
\begin{center}
	\begin{tikzpicture}
\fill(-2,1) circle(2pt);
\fill(-2,0) circle(2pt);
\fill(-1,0) circle(2pt) node[below]{$_{(v,A)}$};
\fill(-1,1) circle(2pt) node[above]{$_{(v,1)}$};
\draw (-2,1)--(-2,0)--(-1,0)--(-1,1);
\draw (-1.5,-1.2) node{${G}$};
\draw(-1.5,0.5) ellipse (37pt and 37pt);

\fill(1,1) circle(2pt) node[above]{$_{(v,1)}$};
\fill(1,0) circle(2pt)node[below]{$_{(v,B)}$};
\fill(2,0) circle(2pt);
\fill(2,1) circle(2pt) ;
\draw (1,1)--(1,0)--(2,0)--(2,1);
\draw (1.5,-1.2) node{${H}$};
\draw(1.5,0.5) ellipse (37pt and 37pt);

\fill(4,1) circle(2pt);
\fill(4,0) circle(2pt);
\fill(5,0) circle(2pt) node[below]{$_{(v,A)}$};
\fill(5.5,1) circle(2pt)  node[above]{$_{(v,1)}$};
\fill(6,0) circle(2pt) node[below]{$_{(v,B)}$};
\fill(7,0) circle(2pt);
\fill(7,1) circle(2pt);
\draw(4,1)--(4,0)--(5,0)--(5.5,1)--(6,0)--(7,0)--(7,1);
\draw (5.5,-1.2) node{${G\cup H}$};
\draw(5.5,0.5) ellipse (65pt and 35pt);

	\end{tikzpicture}
\caption{Graphs $G$, $H$ and $G\cup H$.}
\label{rys_sklej_w}
\end{center}
\end{figure}
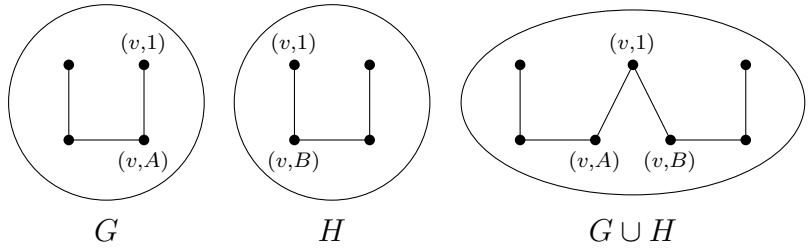

\begin{obs}\label{sklej_k}
Let $G$ be a general corona of a tree and let $(v,1)$ be an external vertex of $G$. If we contract two distinct neighbors of $(v,1)$, then the resulting graph is also a~general corona of a tree.
\end{obs}
\begin{proof}Assume that $G=T\circ \mathcal{P}$ for some tree $T$ and its neighborhood partition~$\mathcal{P}$. Let $(v,A)$ and $(v,B)$ be distinct neighbors of $(v,1)$. Then the graph $G'$, obtained from $G$ by the contraction of $(v,A)$ and $(v,B)$, is a $\mathcal{P'}$-corona of $T$, where $\mathcal{P'}(v)=(\mathcal{P}(v)-\{A,B\})\cup \{A\cup B\}$, and $\mathcal{P'}(x)=\mathcal{P}(x)$ if $x\in V(T)-\{v\}$, see Figure~\ref{rys_sklej_k}.
\end{proof}

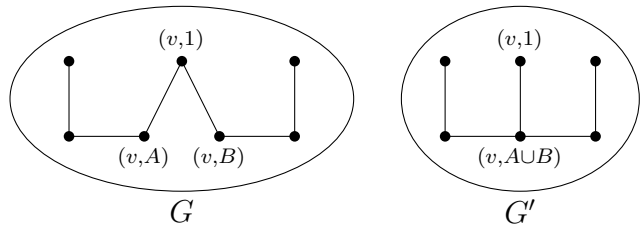
\begin{figure}[htp] 
\begin{center}
	\begin{tikzpicture}

\fill(-1,1) circle(2pt);
\fill(-1,0) circle(2pt);
\fill(0,0) circle(2pt) node[below]{$_{(v,A)}$};
\fill(0.5,1) circle(2pt) node[above]{$_{(v,1)}$};
\draw (-1,1)--(-1,0)--(0,0)--(0.5,1);
\draw (0.5,-1) node{${G}$};
\draw(5,0.5) ellipse (45pt and 35pt);
\fill(1,0) circle(2pt)node[below]{$_{(v,B)}$};
\fill(2,0) circle(2pt);
\fill(2,1) circle(2pt) ;
\draw (0.5,1)--(1,0)--(2,0)--(2,1);
\draw(0.5,0.5) ellipse (65pt and 35pt);
\fill(4,1) circle(2pt);
\fill(4,0) circle(2pt);
\fill(5,0) circle(2pt) node[below]{$_{(v,A\cup B)}$};
\fill(5,1) circle(2pt) node[above]{$_{(v,1)}$};
\fill(6,0) circle(2pt);
\fill(6,1) circle(2pt);
\draw(4,1)--(4,0)--(5,0)--(5,1);
\draw(5,0)--(6,0)--(6,1);
\draw (5,-1) node{${G'}$};
	\end{tikzpicture}
\caption{Graphs $G$ and $G'$.}
\label{rys_sklej_k}
\end{center}
\end{figure}

From Observations~\ref{sklej_w} and~\ref{sklej_k} we immediately have the next observation  (see Figure~\ref{rys_sklej} for an illustration).
\begin{obs}\label{sklej}Let $G$ and $H$ be general coronas of some trees. If they share only one edge such that exactly one of its end vertices  is an external vertex in each of $G$ and $H$, then the union $G\cup H$ is a general corona.
\end{obs}
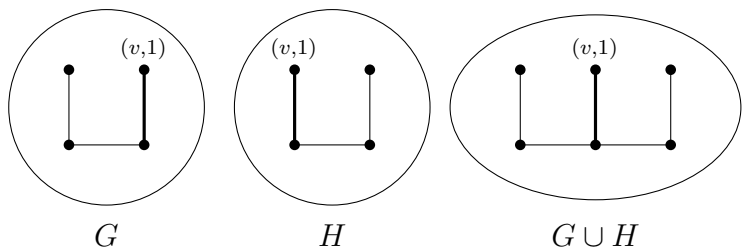
\begin{figure}[htp] 
\begin{center}
	\begin{tikzpicture}
\fill(-2,1) circle(2pt);
\fill(-2,0) circle(2pt);
\fill(-1,0) circle(2pt);
\fill(-1,1) circle(2pt) node[above]{$_{(v,1)}$};
\draw (-2,1)--(-2,0)--(-1,0)--(-1,1);
\draw (-1.5,-1.2) node{${G}$};
\draw(-1.5,0.5) ellipse (37pt and 37pt);

\fill(1,1) circle(2pt) node[above]{$_{(v,1)}$};
\fill(1,0) circle(2pt);
\fill(2,0) circle(2pt);
\fill(2,1) circle(2pt) ;
\draw[very thick](1,0)--(1,1);
\draw[very thick](-1,0)--(-1,1);
\draw[very thick](5,0)--(5,1);

\draw (1,1)--(1,0)--(2,0)--(2,1);
\draw (1.5,-1.2) node{${H}$};
\draw(1.5,0.5) ellipse (37pt and 37pt);

\fill(4,1) circle(2pt);
\fill(4,0) circle(2pt);
\fill(5,0) circle(2pt) ;
\fill(5,1) circle(2pt)  node[above]{$_{(v,1)}$};
\fill(6,0) circle(2pt);
\fill(6,1) circle(2pt);
\draw(4,1)--(4,0)--(5,0)--(5,1);
\draw (5,0)--(6,0)--(6,1);
\draw (5,-1.2) node{${G\cup H}$};
\draw(5,0.5) ellipse (55pt and 35pt);

	\end{tikzpicture}
\caption{Graphs $G$, $H$ and $G\cup H$.}
\label{rys_sklej}
\end{center}
\end{figure}

\section{Trees with domination subdivision number $3$}


The \emph{domination subdivision number} of a graph $G$, denoted by sd$(G)$, is the minimum number of edges which must be subdivided (where each edge can be subdivided at most once) in order to increase the domination number. Since the domination number of the graph $K_2$ does not increase when its edge is subdivided, we consider the subdivision numbers for connected graphs of order at least $3$. The domination subdivision number was defined by Velammal  \cite{vel} and since then it has been widely studied, see \cite{myn, bhv, fav3, fav2,HHH} to mention just a few.

It was shown in \cite{vel} that the domination subdivision number of a tree is either 1, 2 or 3. Let $\mathcal{S}_i$ be the family of trees with domination subdivision number equal to $i$ for $i\in \{1,2,3\}$. Some characterizations of the classes $S_1$ and $S_3$ were given in \cite{myn} and \cite{fav1}, respectively. In particular, the following constructive characterization of $S_3$ was given  in~\cite{fav1}. 

Let the label of a vertex $v$  be denoted by $l(v)$ and $l(v)\in \{A,B\}$. Now, let $\mathcal{F}$ be the family of labeled trees that (i) contains $P_4$, where leaves have label $A$ and support vertices have label $B$, and (ii) is closed under the following two operations, which extend a~labeled tree $T\in \mathcal{F}$ by attaching a~labeled path to a vertex $v\in V(T)$ in such a way that:
\begin{itemize}
\item  If $l(v) = A$, then we add a path $(x,y,z)$ (with labels $l(x)=l(y)= B$ and $l(z)= A$) and an edge $vx$.
\item If $l(v) =B$, then we add a path $(x,y)$ (with labels $l(x)=B$ and $l(y)=A$) and an edge $vx$. 
\end{itemize}

The following characterization of trees belonging to the class $\mathcal{S}_3$ was given in~\cite{fav1}.
\begin{thm}  \label{favaron} The next three statements are equivalent for a tree $T$ with at least three vertices:
\begin{itemize}
\item[$1.$]  $T$ belongs to the class $\mathcal{S}_3$.
\item[$2.$]   $T$ has a unique dominating $2$-packing containing all leaves of $T$.
\item[$3.$]  $T$ belongs to the family $\mathcal{F}$.
\end{itemize}
\end{thm}

Now we are in position to  give a new  characterization of trees belonging to the class $\mathcal{S}_3$. Namely, we shall show that all these graphs precisely are general coronas of trees. 

\begin{lem} \label{war1}
If a tree $T$ is a general corona, then $T$ belongs to $\mathcal{S}_3$.
\end{lem}
\begin{proof}  
From Observation~\ref{ext}, the set of external vertices of $T$ is a dominating \linebreak $2$-packing containing all leaves of $T$ and, consequently, by Theorem~\ref{favaron}, $T\in \mathcal{S}_3$.
\end{proof}

\begin{lem}\label{war2}
If a tree $T$ belongs to $\mathcal{S}_3$, then $T$ is a  general corona. 
\end{lem}

\begin{proof} We use induction on $n$, the number of vertices of a tree. The smallest tree belonging to $\mathcal{S}_3$ is a path $P_4$ and obviously $P_4$ is a $\mathcal{P}$-corona of $P_2$. Let $T\in \mathcal{S}_3$ be a tree on $n$ vertices, $n>4$. We will show that $T$ is a general corona. Let $P=(v_0,v_1,\ldots, v_k)$ be a~longest path in $T$. From the choice of $P$, since $T$ does not have a strong support vertex (by Theorem~\ref{favaron}), it follows that  $k\geqslant 4$ and $d_T(v_1)=2$. We consider two cases: $d_T(v_2)=2$, $d_T(v_2)>2$.\\
{\it Case 1: }$d_T(v_2)=2$. Let $T_1$ and $T_2$  denote  subtrees $T[\{v_0,v_1,v_2,v_3\}]$ and $T-\{v_0,v_1,v_2\}$, respectively. By Theorem~\ref{favaron}, the tree $T$ has a dominating $2$-packing $S$ containing all leaves of $T$ and certainly $\{v_0,v_3\}\subseteq S$. Consequently, $S-\{v_0\}$ is a dominating $2$-packing in $T_2$ containing all leaves of $T_2$. Again by Theorem~\ref{favaron}, the tree  $T_2$ belongs to $S_3$. Thus, by induction, $T_2$ is a~general corona. Since $v_3$ belongs to $S-\{v_0\}$, by Observation~\ref{ext} and Theorem~\ref{favaron}, $v_3\in Ext(T_2)$. Obviously $T_1=P_4$ is a general corona. Because $v_3$ is also an external vertex in $T_1$, and trees $T_1$ and $T_2$ do not share any other vertex, $T=T_1\cup T_2$ is a~general corona by Observation~\ref{sklej_w}.\\
\\
{\it Case 2: }$d_T(v_2)>2$. In this case, again by Theorem~\ref{favaron}, the tree $T$ has a dominating $2$-packing $S$ containing all leaves of $T$. Let $v'$ be the unique neighbor of $v_2$ belonging to $S$. Notice that $v'$ is a leaf if $v_2$ is a support vertex,  otherwise $v'=v_3$. Let $T_1$ and $T_2$ be subtrees $T[\{v_0,v_1,v_2,v'\}]$ and $T-\{v_0,v_1\}$, respectively.  It is easy to observe that $S-\{v_0\}$ is a dominating  $2$-packing in $T_2$ containing all leaves of $T_2$. Now, again by Theorem~\ref{favaron}, the tree  $T_2$ belongs to $S_3$. Thus, by induction, $T_2$ is a~general corona. Since $v'$ belongs to $S-\{v_0\}$, by Observation~\ref{ext} and Theorem~\ref{favaron}, $v'\in Ext(T_2)$.
Certainly $T_1=P_4$ is a~general corona and $v'$ is external vertex in $T_1$. In addition, $T_1$ and $T_2$ share only the edge $v_2v'$. Consequently, by Observation~\ref{sklej}, the tree $T=T_1\cup T_2$ is a general corona.
\end{proof}

Taking into account Observation~\ref{nowa}, Theorem~\ref{favaron}, Lemma~\ref{war1} and Lemma~\ref{war2} we have the summary result.
\begin{thm}\label{tw_gl} Let $T$ be a tree with at least three vertices. Then the following properties are equivalent:
\begin{itemize}
\item[$1.$] The domination subdivision number of $\ T$ is equal to $3$ $\ (i.e.\ T\in \mathcal{S}_3)$.
\item[$2.$]   $T$ has a unique dominating $2$-packing containing all leaves of $\ T$.
\item[$3.$]  $T$ belongs to the family $\mathcal{F}$.
\item[$4.$] $T$ is a general corona of a tree.
\item[$5.$] There exists a tree $T'$ such that $T$ is obtained from the $2$-subdivision $S_2(T')$ by a~sequence of internal contractions.
\item[$6.$] There exists a tree $T'$ such that $T$ is obtained from the corona $T'\circ K_1$ by a~sequence of internal splittings.
\end{itemize}
\end{thm}

\end{document}